\documentclass[a4paper,12pt]{article}
\usepackage{amsmath,amsthm,amssymb}
\usepackage{tabularx,bm,wrapfig}
\usepackage[dvipdfmx]{graphicx}
\newcommand{\GF}{{\mathbb F}}

\newcommand{\Aut}{{\rm Aut}}

\newtheorem{Thm}{Theorem}[section]
\newtheorem{Lem}[Thm]{Lemma}

\theoremstyle{definition}
\newtheorem{Def}[Thm]{Definition}

\begin{document}
\title{An alternative construction of the $G_2(2)$-graph}
\author{Koichi Inoue \thanks{E-mail: k-inoue@teac.paz.ac.jp} \\ \footnotesize{Gunma Paz University, 
Gunma-ken 370-0006, Japan.}} 
\date{}
\maketitle

\begin{flushright}
\tiny{Last modified: 2026/01/04}
\end{flushright}

\paragraph{Abstract.} In this note, we give an alternative construction of the $G_2(2)$-graph from a $U_3(2)$-geometry. 


\setcounter{section}{+0}
\section{Introduction}
The paper~\cite{inoue} has given an alternative approach to construct the Chevalley group $G_2(2)$ with a certain 2-transitive permutation representation of degree 28 as well as a concrete tool to study the structure of it. In this note, using the quite elementary description, we shall give an alternative construction of the $G_2(2)$-graph, which yields an alternative approach of $G_2(2)$ with a permutation representation of degree 36.

\section{Preliminaries}
For definition, notation and known facts of strongly regular graphs and two-graphs, we refer to  Cameron and van Lint~\cite[Chapters 2 and 4]{cameron-lint}
\vspace{1.0\baselineskip}

\noindent
A \textit{graph} $\Gamma$ consists of a finite set of \textit{vertices} together with a set of \textit{edges}, where an edge is a $2$-subset of the vertex-set. The set of all vertices adjacent to a vertex $x$ is denoted by $\Gamma(x)$ and called \textit{neighbours} of $x.$ If $\Gamma(x)$ always has the constant $k$ for each vertex $x$, then $\Gamma$ is said to be $k$-\textit{regular}. The \textit{complement} $\overline{\Gamma}$ of $\Gamma$ has the same vertices as $\Gamma$, but two vertices are adjacent in $\overline{\Gamma}$ if and only if they are not adjacent in $\Gamma$. Let $\{x_1,x_2,\dots,x_n\}$ be the vertex-set of $\Gamma$. The \textit{adjacency matrix} $A$ of $\Gamma$ is the $n\times n$ matrix with $(i,j)$-entry 1 if $x_i$ and $x_j$ are adjacent, 0 otherwise.

For two graphs $\Gamma_1$ and $\Gamma_2,$ we define an $\textit{isomorphism}$ $\sigma$ 
from $\Gamma_1$ onto $\Gamma_2$ to be a one-to-one mapping from the vertices of $\Gamma_1$ onto the vertices of $\Gamma_2$ and the edges of $\Gamma_1$ onto the edges of $\Gamma_2$ such that $x$ is in $e$ if and only if $x^{\sigma}$ is in $e^{\sigma}$ for each vertex $x$ and each edge $e$ of $\Gamma_1$, and then say that $\Gamma_1$ and $\Gamma_2$ are $\textit{isomorphic}$. An \textit{automorphism} of $\Gamma$ is an isomorphism from $\Gamma$ to itself. The set of all automorphisms of $\Gamma$ forms a group, which is called the \textit{automorphism group} of $\Gamma$ and denoted by $\Aut\Gamma$.
 
The graph $\Gamma$ is called a \textit{strongly regular} graph 
with parameters $(n,k,\lambda,\mu)$ if $\Gamma$ is $k$-regular on $n$ vertices, not complete or null, and satisfy that any two adjacent vertices have $\lambda$ common neighbours 
and any two non-adjacent vertices have $\mu$ common neighbours.  

\begin{Lem}\label{lem:2.1}
The complement of a strongly regular graph with parameters $(n,k,\lambda,\mu)$ is also strongly regular graph with parameters $(n,n-k-1,n-2k+\mu-2,n-2k+\lambda)$. 
\end{Lem}

Let $\Gamma$ be a strongly regular graph with parameters $(n,k,\lambda,\mu)$, and $\{x_1,x_2,\dots,x_n\}$ the vertex-set. If $A$ is the adjacency matrix of $\Gamma$, then the $(i,j)$-entry of $A^2$ is the number of vertices adjacent to $x_i$ and $x_j$. This number is $k, \lambda$ or $\mu$ according as $x_i$ and $x_j$ are equal, adjacent or non-adjacent. Thus
\[ A^2=kI+\lambda A+\mu(J-I-A), \tag{$\ast$}\] 
where $I$ is the identity matrix and $J$ is the all-1 matrix. Also, since $\Gamma$ is regular, 
\[ AJ=JA=kJ. \tag{$\ast\ast$}\]
Conversely, a strongly regular graph can be defined as a graph (not complete or null) of which the adjacency matrix satisfies $(\ast)$ and $(\ast\ast)$.

Instead of the ordinary adjacency matrix for a graph with vertex-set $V:=\{x_1,x_2,\dots,x_n\}$, J.~J.~Seidel considered the \textit{Seidel matrix} $S$ whose $(i,j)$-entry is 
\[s_{ij}=\begin{cases}
0 & \text{if $i=j$}, \\
-1 & \text{if $\{x_i,x_j\}$ is an edge}, \\
+1 & \text{if $\{x_i,x_j\}$ is a non-edge}.
\end{cases}\]
The operation of \textit{switching} a graph $\Gamma$ with respect to a subset of $V$ replaces $\Gamma$ by the graph $\Gamma^{'}$ such that all edges between $Y$ and its complement $\overline{Y}$ with non-edges and \textit{vice versa}, leaving edges within $Y$ and $\overline{Y}$ unaltered. Note that switching w.r.t $Y$ and $\overline{Y}$ are the same operation. Furthermore, switching successively w.r.t. $Y_1$ and $Y_2$ is the same as switching w.r.t. the symmetric difference $Y_1 \bigtriangleup Y_2$. Thus the set of all graphs on $V$ falls into equivalence classes (called \textit{switching classes}). Two graphs in the same switching class are called \textit{switching equivalent}.

\begin{Lem}\label{lem:2.2}
Given a graph $\Gamma$, if $\Gamma^{'}$ is the graph obtained by switching $\Gamma$ w.r.t. $Y$, then $\Gamma^{'}$ has the Seidel matrix $DSD$, where $S$ is the Seidel matrix of $\Gamma$ and $D$ is the diagonal matrix whose $(i,i)$-entry is
\[d_{ii}=\begin{cases}
-1 & \text{if $x_i \in Y$}, \\
+1 & \text{if $x_i \notin Y$}.
\end{cases}\]
\end{Lem}  

A \textit{two-graph} consists of finite set $V$ and a set $\mathcal{B}$ of 3-subsets of $V$ such that, for any 4-subset $X$ of $V$, an even number of members of $\mathcal{B}$ belong to $X$. Moreover, A two-graph is said to be \textit{regular} if it is a 2-design (with parameters 2-($v$,3,$\lambda$) for some $\lambda$ ). Given any graph $\Gamma$ with the vertex-set $V$, let $\mathcal{B}$ be the set of 3-subsets of $V$ which contain an odd number of edges of $\Gamma$. Then the incident structure ($V,\mathcal{B}$) is a two-graph, and we call it the \textit{two-graph associated} with $\Gamma$. 

\begin{Lem}\label{lem:2.3}
Graphs $\Gamma_1$ and $\Gamma_2$ on the vertex-set $V$ are associated to the same two-graph if and only if $\Gamma_1$ and $\Gamma_2$ are switching equivalent.
\end{Lem}

\section{The $G_2(2)$-graph}
The Chevalley group $G_2(2)$ has a permutation representation of degree 36 and rank 3. Therefore there is the so-called $rank$ 3 $graph$ which is a strongly regular graph with parameters $(36,14,4,6)$ (e.g., see 
\cite[p.36]{cameron-lint}), and we call it $G_2(2)$-\textit{graph}. Since $G_2(2)$ is a maximal subgroup of the symplectic group $S_6(2)$ (see ATLAS~\cite{atlas}), the graph is geometrically constructed from 36 sub-GH(1,2)'s (i.e., generalized subhexagons of order (1,2)) of the so-called the \textit{split Cayley hexagon} of $order$ (2,2) in an $S_6(2)$-geometry (i.e., a 6-dimensional vector space over the field $\GF_2$ with a non-degenerate alternating form). 
Also, the isomorphism $G_2(2) \simeq U_3(3)$:2 translates into an alternative description of the graph in terms of a $U_3(3)$-geometry (i.e., a 3-dimensional vector space over the field $\GF_9$ with a non-degenerate hermitian form). For more detail on the constructions, see Brouwer and van Maldeghem~\cite[Section 10.8]{bm}, 

\section{The construction}
Let $V$ be a 3-dimensional vector space over the finite field $\GF_{4}$, and $h$ a non-degenerate  hermitian form. Then we consider the $U_3(2)$-geometry $(V,h)$ instead of an $S_6(2)$-geometry, as $V$ is considered as a 6-dimensional vector space over the field $\GF_2$ with the non-degenerate alternating form $s$, where $s(x,y):=h(x,y)+h(y,x)$ for all $x,y \in V$ (see Taylor~\cite[Exercise 10.14]{taylor}), and in turn give an alternative and explicit construction of the $G_2(2)$-graph, which is not appearing elsewhere in the above references. The definitions and the notations of \cite[Preliminaries]{inoue} will be freely used. 
\vspace{1.0\baselineskip}

\noindent
Since $H_0 \cap \tilde{H_0}=\emptyset$ and $|H_{0}|=|\tilde{H_0}|=6$, the set $V_1:= \{x \in V \mid h(x,x)=1\}$ is divided into  
\begin{align*}
&W := \{x \in V_1 \mid [x] \in H_0\} \\
&\text{and} \\
&\tilde{W}:=\{x \in V_1 \mid [x] \in \tilde{H}_0\},
\end{align*}
each consisting of 18 vectors.
\begin{Def}
Define the graph $\Gamma$ whose vertex-set is $V_1$, and two distinct vertices $u$ and $v$ are adjacent whenever $s(u,v)=0$. Then it is well known that $\Gamma$ is an srg(36,15,6,6) with the automorphism group $O_{6}^{-}(2)$:2 (see \cite[Section 3.1.2]{bm}). Let $\Gamma^{\prime}$ be the graph obtained by switching $\Gamma$ with respect to $W$. 
\end{Def}

\begin{Lem}\label{lem:4.2}
$\Gamma^{\prime}$ is $21$-regular
\end{Lem}
\begin{proof}
For $u \in W$, it is easy to see that $\Gamma^{\prime}(u)$ is the disjoint union of 
\[\{x \in W \mid h(x,u)=0,1\} \ \ \ \text{and} \ \ \ \{x \in \tilde{W} \mid h(x,u)=\omega,\bar{\omega}\},\]
each consisting of 9 vectors and 12 vectors, respectively, and so $|\Gamma^{\prime}(u)|=9+12=21$, for each $u \in W$. In a similar way to the above statement, each $u \in \tilde{W}$ has valency 21. 
\end{proof}

\begin{Lem}\label{lem:4.3}
$\Gamma^{\prime}$ is an $\text{srg}(36,21,12,12)$.
\end{Lem}
\begin{proof}
Let $A$ be the usual adjacency matrix of $\Gamma$, and $S$ its Seidel matrix. Then $S=J-I-2A$, so
\[2A=J-I-B, \tag{$1$}\]
where $J$ is the all-1 matrix and $I$ is the identity matrix. Since $\Gamma$ is an srg(36,15,6,6), $A^2=15I+6A+6(J-I-A)=9I+6J$, so
\[(2A)^2=36I+24J. \tag{$2$}\]
By (1) and (2),
\[ (J-I-S)^2=36I+24J. \tag{$3$}\]
Since $S$ has one entry 0, 15 entries $-1$ and 20 entries +1,
\[ SJ=JS=5J. \tag{$4$}\]
By (3) and (4), we obtain
\[ S^2=35I+2S. \tag{$5$}\]
Put $V_1=\{u_1,u_2,\dots,u_{36}\}$. $\Gamma^{\prime}$ has the Seidel matrix $S^{\prime}:=DSD$, where $D:=(d_{ij})$ is the diagonal matrix with 
\[d_{ii}=
\begin{cases}
-1 & \text{if $u_i \in W$,} \\
+1 & \text{if $u_i \notin W$.}
\end{cases}\] 
Then
\[S^{\prime}\mbox{}^{2}=DSD. \tag{$6$}\]
By (5) and (6),
\[S^{\prime}\mbox{}^{2}=35I-2S^{\prime}. \tag{$7$}\]
If $A^{\prime}$ is the usual adjacency matrix of $\Gamma^{\prime}$, then 
\[S^{\prime}=J-I-2A^{\prime}, \tag{$8$}\]
and Lemma~\ref{lem:4.2} shows that
\[ JA^{\prime}=A^{\prime}J=21J. \tag{$9$}\]
By (7), (8) and (9),
\[ (J-I-2A^{\prime})^{2}=35I-2(J-I-2A^{\prime}),\]
and so we obtain
\[ A^{\prime}\mbox{}^{2}=9I+12J=21I+12A^{\prime}+12(J-I-A^{\prime}).\]
Thus $\Gamma^{\prime}$ is an srg(36,21,12,12).
\end{proof}

\begin{Lem}\label{lem:4.4}
$\Aut\Gamma^{\prime}\simeq G_2(2)$.
\end{Lem}
\begin{proof}
E.~Spence~\cite{spence} has classified srgs on 36 vertices. According to the table in \cite[pp.474-496]{spence}, since the two graphs $\Gamma$ and $\Gamma^{\prime}$ belong to the same switching class (No.1), it follows that $\Aut\Gamma^{\prime}\simeq G_2(2)$. 
\end{proof}

\begin{Thm}\label{thm:4.5}
The complementary graph $\overline{\Gamma^{\prime}}$ is an srg$(36,14,4,6)$ and its edges consist of three types as follows:
\begin{itemize}
\item[$\bullet$] $\{u,v\}$ where $u,v \in W$ and $h(u,v)=\omega \ \text{or} \ \bar{\omega};$
\item[$\bullet$] $\{u,v\}$ where $u,v \in \tilde{W}$ and $h(u,v)=\omega \ \text{or} \ \bar{\omega};$
\item[$\bullet$] $\{u,v\}$ where $u \in W, v \in \tilde{W}$ and $h(u,v)=1.$
\end{itemize}
\end{Thm}


\end{document}